\documentclass{article}
\usepackage[utf8]{inputenc}
\usepackage[T1]{fontenc}
\usepackage{lmodern}
\usepackage[english]{babel}

\usepackage{amsmath,amsfonts,amssymb,amsthm}
\usepackage{tikz} 
\usetikzlibrary{calc,decorations.markings,matrix,arrows}
\usepackage{xcolor} 

\usepackage{array} 
\usepackage{enumitem} 
\usepackage{manfnt} 
\usepackage{mathrsfs} 
\usepackage{hyperref} 

\usepackage{stmaryrd} 
\usepackage{mathtools}
\usepackage{textcomp}
\usepackage{xspace}     
\usepackage{graphicx}
\usepackage{pifont}
\usepackage{fancybox}
\usepackage{array}
\usepackage{multirow}
\usepackage{hhline}
\usepackage{tabularx}
\usepackage{graphicx}
\usepackage{fancyvrb}
\usepackage{listings}
\usepackage{pifont}
\usepackage{dsfont}
\usepackage{subfigure}
\usepackage{multicol} 
\usepackage{arydshln} 
\usepackage{fancybox} 
\usepackage{multicol} 
\usepackage{fancybox}
\usepackage[all]{xy}
\usepackage{graphics} 
\usepackage{graphicx} 
\usepackage{pstricks,pst-node} 
\usepackage{tikz} 

\makeatletter

\newcommand\R{\mathbb{R}}

\newcommand\C{\mathbb{C}}

\newcommand\N{\mathbb{N}}

\title{Negative holomorphic bisectional curvature of some bounded domains }
\author{OMAR BAKKACHA }
\date{}

\newtheorem{theorem}{Theorem}[section]
\newtheorem{corollary}{Corollary}[theorem]
\newtheorem{lemme}[theorem]{Lemma}
\newtheorem{proposition}{Proposition}[section]
\theoremstyle{definition}
\newtheorem{definition}{Definition}[section]

\newtheorem{remark}{Remark}[theorem]
\theoremstyle{question}

\newtheorem*{proposition*}{Proposition}
\newtheorem*{theorem*}{Theorem}
\newtheorem*{corollary*}{Corollary}
\begin{document}

\maketitle

\begin{abstract}
    We prove that a bounded domain in $\C^n$ admitting a complete Kähler metric with negatively pinched holomorphic bisectional curvature near the boundary, admits a complete Kähler metric with negatively pinched holomorphic bisectional curvature everywhere. 
    
    \noindent As a consequence we prove that strictly pseudoconvex bounded domains with $\mathcal{C}^2$ boundary and bounded domains with squeezing function tending to 1 at every point of the boundary, admit a complete Kähler metric with negatively pinched holomorphic bisectional curvature everywhere. 
\end{abstract}

\vspace{0.5cm}

\section{Introduction}
\vspace{0.5cm}
 \ \ \ \  In the following text, we study how to extend the negativity and negative pinching of the holomorphic bisectional curvature, from outside a compact to the whole space, in the special case of bounded domains in $\C^n$.
 
 \noindent In that vein, we start, in the first part, by recalling some useful facts such as the Schwarz-Yau lemma, and giving a more straightforward proof to H. Wu's theorem [Wu73] (see Theorem $2.4$) using the Chern-Lu formula.
 \newline
 
 \noindent We then prove the main theorem of this text (see Theorem $\mathbf{3.6}$) 
 \begin{theorem*} Let $\Omega \subset \subset \C^n$ be a bounded domain, and $K\subset \Omega$ a compact subset. If $\Omega$ admits a complete Kähler metric of negatively pinched holomorphic bisectional curvature on $\Omega\setminus K$, then there exists a complete Kähler metric of negatively pinched holomorphic bisectional curvature on $\Omega$.
\end{theorem*}

\noindent To that purpose we introduce the notion of spherical comparison of two Hermitian metrics (see Definition $3.1$), which appears to be of importance in order to generalize H.Wu's inequality [Wu73] from holomorphic sectional curvature, to holomorphic bisectional curvature (see Theorem $3.2$). The new metric that is constructed is given by the sum of the basic metric with a constant times the Bergman metric of a ball containing $\Omega$ in its interior (which we denote by $b_\Omega^R$ see notations and definitions bellow). The constants that pinch the curvature of the new metric are essentially expressed in terms of the spherical comparisons of the two metrics that we add together.  
\newline

\noindent We then investigate some consequences of the preceding theorem when adding some convexity conditions on $\Omega$. 
\noindent The first consequence is (see Corollary $\mathbf{4.2.1}$)
\begin{corollary*}
    Let $\Omega \subset \subset \C^n$ be a strictly pseudoconvex  bounded domain with $\mathcal{C}^2$ boundary, then there exists a complete Kähler metric on $\Omega$ with negatively pinched holomorphic bisectional curvature on $\Omega$.
\end{corollary*}

\noindent Which is obtained by taking the Bergman metric of $\Omega$ as the basic metric and using a classical result due to Klembeck [Kle78] and Kim and Yu [KY96] (see Theorem 4.2).
\newline

\noindent The second consequence is (see Corollary $\mathbf{4.4.1}$)

\begin{corollary*}
Let $\Omega\subset \subset \C^n$ be a bounded domain, if the squeezing function tends to $1$ at every point of the boundary, then there exists a complete Kähler metric on $\Omega$ with negatively pinched holomorphic bisectional curvature on $\Omega$.
\end{corollary*}
\noindent Which is obtained by taking the Kähler-Einstein metric of $\Omega$ (which exists due to the assumption on the squeezing function) as the basic metric (see notations and definitions bellow for the definition of the squeezing function ). The main ingredient is a theorem due to Gontard [Gon18] (see Theorem 4.4), and the combination of some results concerning the squeezing function that can be found in [Yeu09] and [DGZ12].
\newline

We shall fix some notations and definitions once and for all.


\paragraph{Notations and definitions:}
\
\vspace{0.5cm}
\\Let $M$ be a complex manifold, and $h$ a Hermitian metric on $M$.

\begin{itemize}
\item If $u$ is real smooth function on $M$, and $\left\{ \theta_i \right\}$ a unitary basis of 1-forms of type $(1,0)$ with respect to $h$, defined on an open subset of $M$ then we define $u_i$ to be the local functions that appear in the decomposition   
\begin{center}
     $\displaystyle du = \underset{i}{\sum} u_{i}\theta_i + \Bar{u}_{i}\Bar{\theta}_i$.
\end{center}
And we define the \textit{complex Laplacian} 
\begin{center}
    $\displaystyle \Delta_c u := \underset{i}{\sum} u_{ii}$.
\end{center}
Where $u_{ij}$ are the local functions that appear in the decomposition
\begin{center}
    $  \displaystyle -d\partial u = \underset{i}{\sum} d(u_{i}\theta_i) = \partial\Bar{\partial}u = \underset{i,j}{\sum} u_{ij} \theta_i\wedge \Bar{\theta}_j$.
\end{center}
Finally we define the real Laplacian (which is the Laplace-Beltrami operator of $h$ regarded as a Riemmanian metric) as 
\begin{center}
    $\Delta = 2\Delta_c$.

\end{center}    
    \item We denote by $R_h$ the curvature tensor of $h$, which in some local holomorphic coordinate $\left\{ z_1,\ldots, z_n\right\}$ centered at $p\in M$ has the expression
    \begin{center}
        
    $\displaystyle  R_{i\Bar{j}k\Bar{l}} = - \frac{\partial^2}{\partial z_k \partial \Bar{z}_l}h_{i\Bar{j}}+ \underset{p, \Bar{q}}{\sum}  h^{p\Bar{q}}\frac{\partial}{\partial z_k }h_{i\Bar{q}}\frac{\partial}{\partial \Bar{z}_l}h_{p\Bar{j}}$
   \end{center}
   where $h_{i\Bar{j}}$ is the expression of $h$ in those coordinates and $h^{i\Bar{j}}$ the inverse matrix of $h_{i\Bar{j}}$. Hence for $X,Y,Z,T \in T_p^{1,0}M $
   \begin{center}
       $\displaystyle R_h(X,\overline{Y},Z,\overline{T}) =  \underset{i,j,k,l}{\sum} R_{i\Bar{j}k\Bar{l}} X_i\overline{Y}_jZ_k\overline{T}_l $.
   \end{center}
   \item For $p\in M$ and  $X,Y \in T_p^{1,0}M$ we denote by 
\begin{center}
    $\displaystyle HSC(h)(X) = \frac{R_h(X,\overline{X},X,\overline{X})}{h(X,\overline{X})^2}$
\end{center}
the \textit{holomorphic sectional curvature} of the complex line generated by the vectors $X$, and
\begin{center}
    $\displaystyle HBC(h)(X,Y) = \frac{R_h(X,\overline{X},Y,\overline{Y})}{h(X,\overline{X})h(Y,\overline{Y})}$
\end{center}
the  \textit{holomorphic bisectional curvature} of the complex plane generated by the vectors $X$ and $Y$.

\item We denote by $S_h$ the sphere bundle of $h$, and say that 
\begin{center}
    $HBC(h) \leq C$ for some constant $C\in \R$ 
\end{center}
if 
\begin{center}
    $\underset{X,Y\in S_h}{\sup} HBC(h)(X,Y) \leq C$.
\end{center}
We say that $h$ is \textit{of negatively pinched holomorphic bisectional curvature} if 
\begin{center}
    $-C'\leq HBC(h) \leq -C$ for some positive constants $C,C'>0$.
\end{center}
\item We say that a Hermitian metric g is \textit{Kähler-Einstein} if it is Kähler and if there exists a constant $c\in\R$ such that for any $p\in M$ and for $X,Y\in T_p^{1,0}M$
\begin{center}
    $Ric_g(X,\overline{Y}) = cg(X,\overline{Y}) $
\end{center}
where $Ric_g(\cdot,\cdot)$ stands for the Ricci tensor that is 
\begin{center}
    $Ric_g(X,\overline{Y})= Trace(V\mapsto R_g(V,X)\overline{Y}) $ for any $X,Y \in T_p^{1,0} M$.
\end{center} 
\item For $p\in M$ and $X \in T^{1,0}_p M$ we denote by $\mathfrak{K}_M (p, X )$ the \textit{Kobayashi pseudo-metric} evaluated at $X$ that is 
    \begin{center}
        \leftskip=0cm$\mathfrak{K}_M (p, X) := \inf \left\{\left| \xi \right|_{\C} | \ \xi\in \C, \text{there is} \ \phi \in \mathrm{Hol}(\mathbb{D},M) \ \text{with} \ \phi(0) = p, d_0\phi(\xi) = X \right\}$.\rightskip=2cm
    \end{center}

\end{itemize}

\noindent Let $\Omega \subset\subset \C^n$ be a bounded domain.
\begin{itemize}
    \item We denote by $b_\Omega$ the Bergman metric on $\Omega$ which is defined as follows.
    \\ Let 
    \begin{center}
        $\displaystyle A^2(\Omega) := \left\{ f\in \mathscr{O}(\Omega) \ \mathrel{}\middle|\mathrel{} \ \int_\Omega \left| f\right|^2 < \infty \right\}$
    \end{center}
    be the infinite dimensional Hilbert subspace of $L^2(\Omega)$  of holomorphic $L^2$ functions on $\Omega$. Let $\left\{\varphi_i \right\}_{i\in \N}$ be an orthonormal basis of $A^2(\Omega)$ for the $L^2$ natural Hermitian product.
    \\ Let 
    \begin{center}
        $\displaystyle K(z,z')=  \underset{i=0}{\overset{\infty}{\sum}}\varphi_i(z)\overline{\varphi_i(z')}$ for $z,z'\in \Omega$
    \end{center}
     be the Bergman kernel (which is independent of the choice of the orthonormal basis). The Bergman metric of $\Omega$ is then given by
     \begin{center}
         $\displaystyle b_{\Omega i\overline{j}} := \frac{\partial^2 }{\partial z^i\partial \Bar{z}^j}\log K(z,z)$.
     \end{center}
      \item Let $R>0$ be a positive number such that $\Omega \subset\subset B(0,R)$ where $B(0,R)$ stands for the ball in $\C^n$ of radius $R$. We denote by $b^R_\Omega$ the restriction to $\Omega$ of the Bergman metric of $B(0,R)$.
      \item We say that $\Omega \subset\subset \C^n$ is \textit{pseudoconvex} if for any compact subset $K\subset\subset \Omega$ the subset 
      \begin{center}
          $\displaystyle \hat{K}_\Omega := \left \{ z\in \Omega \ \mathrel{}\middle|\mathrel{}  \  \left | f(z) \right | \leq \underset{K}{\sup}\left | f \right | \ \text{for any}  \ f\in \mathscr{O}(\Omega) \right \}$
      \end{center}
      is compact in $\Omega$.
      \item We say that a bounded domain with $\mathcal{C}^2$ boundary $\Omega \subset \subset \C^n$ is \textit{strictly pseudoconvex } if there exist some neighborhood $U$ of $\overline{\Omega}$ and a  $\mathcal{C}^2$ defining function $\rho : U \rightarrow \R$ of $\Omega$ ( i.e. a $\mathcal{C}^2$ function such that $\Omega = \left \{ \rho < 0 \right \}$ and $\partial\Omega = \left \{ \rho = 0 \right \}$ ), such that for any $z\in \partial \Omega$ and any $X\in  T^{1,0}_z\C^n$ satisfying 
      \begin{center}
          $\displaystyle \underset{i}{\sum} \frac{\partial\rho}{\partial z_i}(z)X_i = 0$
      \end{center}
      one has
      \begin{center}
           $\displaystyle \underset{i,j}{\sum} \frac{\partial^2\rho}{\partial z_i \partial \overline{z}_j}(z)X_i\overline{X}_j > 0$.
      \end{center}
      \item  Let $z\in \Omega$, $B(0,1)$ the unit ball in $\C^n$ and consider the set 
      \begin{center}
          $\displaystyle \mathcal{F}^\Omega_z := \left \{ f: \Omega \rightarrow B(0,1) \mathrel{}\middle|\mathrel{}  f \ \text{is a  holomorphic embedding and } f(z)=0 \right \}$
      \end{center}
      we define the \textit{squeezing function of} $\Omega$ as the function 
      \begin{center}
          $\displaystyle s_\Omega(z) := \underset{f\in \mathcal{F}^\Omega_z }{\sup} \sup \left \{ r>0 \mathrel{}\middle|\mathrel{}  B(0,r) \subset f(\Omega) \right \} $.
      \end{center}
\end{itemize}

\section{Schwarz-Yau lemma and Chern-Lu Formula}
\subsection{Schwarz-Yau lemma}
Let $(M,g)$ and $(N,h)$ be two Hermitian manifolds of complex dimension respectively $m$ and $n$. Let $ \left\{ \theta_i \right\}$ (resp. $ \left\{ \omega_\alpha \right\}$) be a unitary basis of 1-forms of type $(1,0)$ defined on an open subset of $M$ (resp. N) that is
\begin{center}
    $\displaystyle g= \underset{i=1}{\overset{m}{\sum}}\theta_i \otimes \overline{\theta}_i $ and \ $\displaystyle h= \underset{\alpha=1}{\overset{n}{\sum}}\omega_\alpha \otimes \overline{\omega}_\alpha$
\end{center}
Let $f: M\rightarrow N$ be a holomorphic map, one has 
\begin{center}
    $\displaystyle f^*\omega_\alpha= \underset{i=1}{\overset{m}{\sum}}a_{\alpha i}\theta_i$ and  $\displaystyle f^*\overline{\omega}_\alpha= \underset{i=1}{\overset{m}{\sum}}\overline{a}_{\alpha i}\overline{\theta}_i$
\end{center}
hence
\begin{center}
    $\displaystyle f^*h= \underset{\alpha=1}{\overset{n}{\sum}}f^*(\omega_\alpha \otimes \overline{\omega}_\alpha) = \underset{\alpha,i,j}{\sum}a_{\alpha i}\overline{a}_{\alpha j} \theta_i \otimes \overline{\theta}_j$.
\end{center}
We call \textit{the norm function} of $f$ \textit{with respect to the metrics} $h$ and $g$ the global real smooth function on $M$
\begin{center}
    $\displaystyle u_h = \underset{\alpha,i}{\sum} a_{\alpha i}\overline{a}_{\alpha i} $.
\end{center}
Moreover denote by $u_{h_i}$ the local functions that appear in the decomposition  
\begin{center}
    $\displaystyle du_h = \underset{i}{\sum} u_{h_i}\theta_i + \Bar{u}_{h_i}\Bar{\theta}_i$
\end{center}
Notice that
\begin{equation}
        f^*h\leq u_h g
\end{equation}

\
\\and that, given $p\in M$ and $\left\{e_i \right\}$ a unitary basis of $T_p^{1,0}M$ associated to the unitary basis of 1-forms $ \left\{ \theta_i \right\}$ (i.e. $\theta_i(e_j)= \delta_{ij})$ we have
\begin{equation}
    u_h(p) = \underset{i}{\sum}h\left (d_pf e_i, \overline{d_pf e_i} \right ).
\end{equation}
Finally denote by $S_{\alpha \Bar{\beta} \gamma \Bar{\eta}}$ the curvature tensor of $(N,h)$, $R_{i\Bar{j}}$ the Ricci tensor of $(M,g)$ and by $\Delta_c$ the complex Laplacian operator on $M$ with respect to the metric $g$ (see notations above). We can now state the Chern-Lu formula [Lu68].

\begin{theorem}[Chern-Lu Formula]
With the preceding notations, and when $u_h >0$ we have

\begin{center}
    $\displaystyle u_h \Delta_c \log u_h =  \frac{1}{2}\underset{\alpha, i , k}{\sum}a_{\alpha i} \overline{a}_{\alpha k}R_{i\Bar{k}} - \frac{1}{2}\underset{ i , k}{\sum} \underset{\alpha, \beta, \gamma , \eta}{\sum}a_{\alpha i} \overline{a}_{\beta i}a_{\gamma k} \overline{a}_{\eta k}S_{\alpha \Bar{\beta} \gamma \Bar{\eta}}  $.
\end{center}
\end{theorem}
\noindent Recall the Riemannian maximum principle due to Omori and Yau [Omo67], [Yau75].

\begin{theorem}[Omori/ Yau]
Let $(M,g)$ be a complete Riemannian manifold with Ricci curvature bounded from below, let $f$ be a real $C^2$ function on $M$ bounded from above, then there exists a sequence $p_k \in M$ such that 
\begin{center}
    $\displaystyle \underset{k\rightarrow\infty}{\lim} \left\|\nabla f(p_k)\right\| =0  $, $\displaystyle \underset{k\rightarrow\infty}{\lim}\sup \ \Delta f(p_k) \leq 0$ and $\displaystyle \underset{k\rightarrow\infty}{\lim}f(p_k) = \underset{M}{\sup} \ f$.
\end{center}
\end{theorem}
\noindent As a consequence of the Chern-Lu formula and the Riemannian maximum principle one obtains the general Schwarz lemma due to Yau [Yau78].
\begin{theorem}[Schwarz-Yau lemma]
Let $(M,g)$ be a complete Kähler manifold with Ricci curvature bounded from below by a negative constant $-C$. Let $(N,h)$ be a Hermitian manifold whith holomorphic bisectional curvature bounded from above by a negative constant $-A$. Then every holomorphic map $f: M \rightarrow N$ satisfies  
\begin{center}
    $\displaystyle f^*h \leq \frac{C}{A} g$.
\end{center}

\end{theorem}

\subsection{A formula for holomorphic (bi)sectional curvature}

\noindent In what follows we deduce some explicit formulas for holomorphic sectional and bisectional curvature using Chern-Lu formula, and give a more direct proof of Wu's inequality for holomorphic sectional curvature.
\\\\
Let $(M,h)$ be a Hermitian manifold, let $p\in M$,  $ 0 \neq X\in T_p^{1,0}M$, $r>0$ and $f:\Delta_r \rightarrow M$ (where $\Delta_r$ stands for the disc centered at $0$ of radius $r$ ) a holomorphic map such that $f(0) = p$ and $\displaystyle d_0f \frac{\partial}{\partial z} = X$ . \\Since $X\neq 0$ we apply the Chern-Lu formula for $M$ endowed with $h$ and the disc endowed with the Euclidean metric.
With the same notations as before 
\begin{center}
    $\displaystyle u_h(0) \Delta_c \log u_h (0) =   - \frac{1}{2} R_h(X,\overline{X},X,\overline{X})  $.
\end{center} 
Remark that $u_h(0) = \left \| X \right \|_h^2$ we hence have 
\begin{equation}
     \displaystyle HSC(h)(X) = \frac{-2}{\left \| X \right \|_h^2}\frac{\partial}{\partial z \partial\overline{z}} \log u_h (0)   . 
\end{equation}
Let $0 \neq X_1 , X_2 \ \in T_p^{1,0}M$, $r>0$ and $f: \Delta_r^2 \rightarrow M$ (where $\Delta_r^2$ is the bidisc centered at $0$ of biradius $r$) a holomorphic map such that $\displaystyle d_0f \frac{\partial}{\partial z_1} = X_1$ and $\displaystyle d_0f \frac{\partial}{\partial z_2} = X_2$.
\\Let $\iota_1, \iota_ 2 :\Delta_r \rightarrow \Delta_r^2$ be the natural injections, and $p_1, p_2 : \Delta_r^2 \rightarrow \Delta_r$ the natural projections. 
\\Denote by $u_h$, $u_h^1$ and $u_h^2$ the norm functions of $f$, $f\circ \iota_1$ and $f\circ \iota_2$ with respect to the metrics  $h$ and the euclidean metrics on $\Delta_r^2$ and $\Delta_r$, applying the Chern-Lu formula  
\begin{center}
    $-\displaystyle u_h(0,0) \left (  \frac{\partial}{\partial z_2 \partial\overline{z}_2} + \frac{\partial}{\partial z_1 \partial\overline{z}_1} \right ) \log u_h (0,0) $
    \vspace{0,2cm}
    \\$\displaystyle =  R_h(X_1,\overline{X_1},X_2,\overline{X_2}) + \frac{1}{2}R_h(X_1,\overline{X_1},X_1,\overline{X_1}) + \frac{1}{2}R_h(X_2,\overline{X_2},X_2,\overline{X_2})$
    \vspace{0,2cm}
    \\$-\displaystyle u_h^1(0) \frac{\partial}{\partial z \partial\overline{z}} \log u_h^1 (0) = \frac{1}{2}R_h(X_1,\overline{X_1},X_1,\overline{X_1}) $
    \vspace{0,2cm}
    \\$-\displaystyle u_h^2(0) \frac{\partial}{\partial z \partial\overline{z}}\log u_h^2 (0) = \frac{1}{2}R_h(X_2,\overline{X_2},X_2,\overline{X_2}) $.
\end{center}
Denote by $\Tilde{u}_h^1$ (resp. $\Tilde{u}_h^2$) the real functions $u_h^1\circ p_1$ (resp. $u_h^2\circ p_2 $), we deduce
\begin{center}
    $-\displaystyle u_h^1(0) \frac{\partial}{\partial z_1 \partial\overline{z}_1}\log \Tilde{u}_h^1 (0,0) = \frac{1}{2}R_h(X_1,\overline{X_1},X_1,\overline{X_1}) $
    \vspace{0,2cm}
    \\$-\displaystyle u_h^2(0) \frac{\partial}{\partial z_2 \partial\overline{z}_2}\log \Tilde{u}_h^2(0,0) = \frac{1}{2}R_h(X_2,\overline{X_2},X_2,\overline{X_2}) $.
\end{center}
Hence 
\begin{center}
    $R_h(X_1,\overline{X_1},X_2,\overline{X_2})$
    \vspace{0,2cm}
    \\$\displaystyle = u_h^2(0) \frac{\partial}{\partial z_2 \partial\overline{z}_2}\log \Tilde{u}_h^2 (0,0) + u_h^1(0) \frac{\partial}{\partial z_1 \partial\overline{z}_1}\log \Tilde{u}_h^1 (0,0) $
    \vspace{0,2cm}
    \\$ \displaystyle -u_h(0,0) \left (  \frac{\partial}{\partial z_2 \partial\overline{z}_2} + \frac{\partial}{\partial z_1 \partial\overline{z}_1} \right ) \log u_h (0,0)  $.

\end{center}

Denote by $HBC(h)(X_1,X_2)$ the holomorphic bisectional curvature of $h$ of the complex plane generated by the vectors $\displaystyle \frac{X_1}{\left \| X_1 \right \|_h}$ et $\displaystyle \frac{X_2}{\left \| X_2 \right \|_h}$, remark finally that
\begin{center}
$\displaystyle  \frac{\partial}{\partial z_1 \partial\overline{z}_1}\log \Tilde{u}_h^1 (0,0) = \Delta_c \log \Tilde{u}_h^1 (0,0)$ and $\displaystyle  \frac{\partial}{\partial z_2 \partial\overline{z}_2}\log \Tilde{u}_h^2 (0,0) = \Delta_c \log \Tilde{u}_h^2 (0,0)$. 
\end{center}
We hence obtain  

\begin{center}
       $HBC(h)(X_1,X_2)$
       \vspace{0,2cm}
    \\$\displaystyle =\frac{1}{\left \| X_1 \right \|_h^2} \Delta_c \log \Tilde{u}_h^2 (0,0) + \frac{1}{\left \| X_2 \right \|_h^2} \Delta_c\log \Tilde{u}_h^1 (0,0) $
    \vspace{0,2cm}
    \\$ \displaystyle -\frac{\left \| X_1 \right \|_h^2 + \left \| X_2 \right \|_h^2}{\left \| X_1 \right \|_h^2\left \| X_2 \right \|_h^2} \Delta_c \log u_h (0,0)   $.
       
\end{center}
As a consequence of formula $(3)$ we provide a more straightforward proof of Wu's theorem [Wu73].  
\begin{theorem}[H.Wu]
Let $M$ be a complex manifold, $g$ and $h$ two Hermitian metrics on $M$, such that their holomorphic sectional curvature $HSC(g)$ and $HSC(h)$ satisfies $HSC(g) \leq - K_1$ and $HSC(h)\leq -K_2$ for $K_1$ and $K_2$ some positive constants, then 
\begin{center}
    $\displaystyle HSC(g+h) \leq -\frac{K_1K_2}{K_1+K_2}$.
\end{center}
\end{theorem}
\begin{proof}
Let $p\in M$ and $X\in T_p^{1,0}M$, $r>0$, and $f:\Delta_r \rightarrow M$ a holomorphic map $f(0) = p$ and $\displaystyle d_0f \frac{\partial}{\partial z} = X$, applying formula $(3)$ of the holomorphic sectional curvature of  $h$, $g$ and $h+g$ we obtain
\begin{center}
    $\displaystyle -2\frac{\partial}{\partial z \partial\overline{z}} \log u_h (0) =    u_h(0) HSC(h)  $ 
\end{center}
and 
\begin{center}
     $\displaystyle -2\frac{\partial}{\partial z \partial\overline{z}} \log u_g (0) =    u_g(0) HSC(g)  $
\end{center}
and 
\begin{center}
    $\displaystyle -2\frac{\partial}{\partial z \partial\overline{z}} \log u_{h+g} (0) =    (u_h+ u_g)(0) HSC({h+g})  $.
\end{center}
Hence 
\begin{center}
    $\displaystyle u_hu_g(u_h+u_g)\left ( u_h^2HSC(h) + u_g^2HSC(g) - (u_h+u_g)^2HSC(h+g)  \right )$
    \vspace{0.2cm}
    \\$\displaystyle = 2 \left |u_h(0)\frac{\partial u_g}{\partial z}(0) -u_g(0)\frac{\partial u_h}{\partial z}(0)   \right |^2 \geq 0$.
\end{center}
Thus 
\begin{center}
    $\displaystyle HSC(h+g) \leq \frac{1}{(u_h(0)+u_g(0))^2}\left ( u_h^2HSC(h) + u_g^2HSC(g) \right )$
    \vspace{0.2cm}
    \\$\displaystyle \leq -\frac{1}{(u_h(0)+u_g(0))^2}\left ( u_h^2K_1 + u_g^2K_2 \right )$
    \vspace{0.2cm}
    \\$\displaystyle \leq -\frac{K_1K_2}{K_1+K_2}$.

\end{center}
\end{proof}

\section{Negatively pinched holomorphic bisectional curvature}

\subsection{Wu type inequality for holomorphic bisectional curvature}
In order to obtain a Wu type inequality for holomorphic bisectional curvature, we introduce the following quantity.

\begin{definition}[Spherical comparison] Let $h$ and $g$ be two Hermitian metrics on $M$, we define the \textit{spherical comparison} of $h$ with respect to $g$ as the quantity 
\begin{center}
    $C(h,g) = \underset{X\in S_{h+g}}{\inf} \left \| X \right \|^2_{h}$.
\end{center}
\end{definition}
\noindent The following lemma can be found in [Wu73] and will be usefull for the theorem that follows.
\begin{lemme}[H.Wu]
    Let $h$ be a Hermitian metric on $M$, $p\in M$ and $X\in T^{1,0}_pM$ a unit tangent vector. Then there exist a neighborhood $V \subset M$ of $p$ and coordinates $\left\{ z^1, \ldots, z^n\right\}$ on $V$ centered at $p$ such that
 \begin{enumerate}
     \item $\displaystyle X = \frac{\partial}{\partial z^1}$
     \item $\displaystyle h_{i\Bar{j}}(p) = \delta_{i\Bar{j}}$
     \item $\displaystyle \frac{\partial h_{i\Bar{j}}}{\partial z^1}(p)= \frac{\partial h_{i\Bar{j}}}{\partial \overline{z}^1}(p) = 0$
 \end{enumerate} 
Where $h_{i\Bar{j}}$ stands for the expression of $h$ in the coordinates $\left\{ z^1, \ldots, z^n\right\}$
\end{lemme}
\begin{theorem}
    Let $g$ and $h$ two Hermitian metrics on $M$, such that $HBC(g) \leq - B_1$ and $HBC(h)\leq -B_2$ for some positive constants $B_1$ and $B_2$ then  
    \begin{center}
        $HBC(h+g) \leq  -C(g,h)^2B_1 -C(h,g)^2B_2 $
    \end{center}
\end{theorem}
\begin{proof}
Let $p\in M$ and $X,Y \in T^{1,0}_pM$ such that
\begin{center}
    $\left \| X \right \|_{h+g} = \left \| Y \right \|_{h+g} = 1$.
\end{center}
Let $\left\{ z^1, \ldots, z^n\right\}$ be coordinates centered at $p\in M$ satisfying conditions 1., 2. and 3. of lemma $3.1$ for the metric $h+g$ and the tangent vector $\displaystyle X$. 
\\Denote by $h_{i\Bar{j}}$ (resp. $g_{i\Bar{j}}$) the expression of $h$ (resp. $g$) in those coordinates. We have at $p \in M$
\begin{center}
    $\displaystyle R_{h+g}(X,\overline{X},Y,\overline{Y}) = - \underset{i,\bar{j}}{\sum} \frac{\partial^2}{\partial z_1 \partial \Bar{z}_1}(h_{i\Bar{j}}+ g_{i\Bar{j}})Y_i\overline{Y}_j$
    $\displaystyle R_{h}(X,\overline{X},Y,\overline{Y}) = \underset{i,\bar{j}}{\sum} \left (  - \frac{\partial^2}{\partial z_1 \partial \Bar{z}_1}h_{i\Bar{j}}+ \underset{p, \Bar{q}}{\sum}  h^{p\Bar{q}}\frac{\partial}{\partial z_1 }h_{i\Bar{q}}\frac{\partial}{\partial \Bar{z}_1}h_{p\Bar{j}}\right )Y_i\overline{Y}_j $
    $\displaystyle R_{g}(X,\overline{X},Y,\overline{Y}) = \underset{i,\bar{j}}{\sum} \left (  -  \frac{\partial^2}{\partial z_1 \partial \Bar{z}_1}g_{i\Bar{j}}+ \underset{p, \Bar{q}}{\sum}  g^{p\Bar{q}}\frac{\partial}{\partial z_1 }g_{i\Bar{q}}\frac{\partial}{\partial \Bar{z}_1}g_{p\Bar{j}}\right )Y_i\overline{Y}_j $ .
\end{center} 
Hence
\begin{center}
       
    $\displaystyle R_{h+g}(X,\overline{X},Y,\overline{Y}) - \displaystyle R_{h}(X,\overline{X},Y,\overline{Y}) - \displaystyle R_{g}(X,\overline{X},Y,\overline{Y})$
    \\$\displaystyle =- \underset{i,\bar{j}}{\sum}\left (  \underset{p, \Bar{q}}{\sum}  g^{p\Bar{q}}\frac{\partial}{\partial z_1 }g_{i\Bar{q}}\frac{\partial}{\partial \Bar{z}_1}g_{p\Bar{j}} + h^{p\Bar{q}}\frac{\partial}{\partial z_1 }h_{i\Bar{q}}\frac{\partial}{\partial \Bar{z}_1}h_{p\Bar{j}}\right ) Y_i\overline{Y}_j $
    \\$\displaystyle =- \underset{p, \Bar{q}}{\sum} \left (    g^{p\Bar{q}}\underset{i}{\sum}\frac{\partial}{\partial z_1 }g_{i\Bar{q}} Y_i\underset{\bar{j}}{\sum}\frac{\partial}{\partial \Bar{z}_1}g_{p\Bar{j}}\overline{Y}_j + h^{p\Bar{q}}\underset{i}{\sum}\frac{\partial}{\partial z_1 }h_{i\Bar{q}}Y_i\underset{\bar{j}}{\sum}\frac{\partial}{\partial \Bar{z}_1}h_{p\Bar{j}}\overline{Y}_j\right )  $ .
\end{center}
Denote by 
\begin{center}
    $\displaystyle A_p := \underset{j}{\sum}\frac{\partial}{\partial \Bar{z}_1 }g_{p\Bar{j}} \overline{Y}_j $ and $\displaystyle B_p := \underset{j}{\sum}\frac{\partial}{\partial \Bar{z}_1 }h_{p\Bar{j}} \overline{Y}_j$ .
\end{center}
Since $h_{i\Bar{j}}$ and $g_{i\Bar{j}}$ are Hermitian, we have  
\begin{center}
    $\displaystyle \overline{A}_q = \underset{i}{\sum}\frac{\partial}{\partial z_1 }g_{i\Bar{q}} Y_i $ and 
    $\displaystyle \overline{B}_q = \underset{i}{\sum}\frac{\partial}{\partial z_1 }h_{i\Bar{q}} Y_i$ .
\end{center}
Thus

 \begin{center}
    $\displaystyle R_{h+g}(X,\overline{X},Y,\overline{Y}) - \displaystyle R_{h}(X,\overline{X},Y,\overline{Y}) - \displaystyle R_{g}(X,\overline{X},Y,\overline{Y})$
    \\$=\displaystyle - \left (  \underset{p, \Bar{q}}{\sum} g^{p\Bar{q}}A_p\overline{A}_q + \underset{p, \Bar{q}}{\sum} h^{p\Bar{q}}B_p\overline{B}_q \right )  $ .
\end{center} 
Since the inverse matrices $g^{p\Bar{q}}$ and $h^{p\Bar{q}}$ are positive definite we have
\begin{center}
    $\displaystyle R_{h+g}(X,\overline{X},Y,\overline{Y}) \leq \displaystyle R_{h}(X,\overline{X},Y,\overline{Y}) + \displaystyle R_{g}(X,\overline{X},Y,\overline{Y})$ .
\end{center}
Hence
\begin{equation}
    HBC(h+g)(X,Y) \leq \left \| X \right \|^2_{h}\left \| Y \right \|^2_{h}HBC(h)(X,Y) + \left \| X \right \|^2_{g}\left \| Y \right \|^2_{g}HBC(g)(X,Y) .  
\end{equation}
Hence
\begin{center}
    $HBC(h+g) \leq -C(g,h)^2B_1 -C(h,g)^2B_2 $. 
\end{center}
\end{proof}
\noindent The following corollary will be useful for proving the main theorem.
\begin{corollary}
   Let $g$ and $h$ two Hermitian metrics on $M$, and $K\subset M$ a compact subset such that $HBC(g) \leq - B_1$ on $M\setminus K$ and $HBC(h)\leq -B_2$ on $M$ for some positive constants $B_1$ and $B_2$. Then there exist some non negative constants $C_0,C_1\geq0$ such that 
\begin{center}
    $HBC(C_0h+g) \leq -C_1$.
\end{center}
\end{corollary}
\begin{proof}
Let $S_g|_K$ be the restriction of the sphere bundle, for the metric $g$, to the compact $K$, and  $A_0 := \underset{X,Y\in S_g|_K }{\sup} \ HBC(g)(X,Y)$. If $A_0<0$ the corollary is a direct consequence of theorem $3.2$. Thus we can suppose that $A_0\geq 0$. 
\\Denote by
\begin{center}
    $\displaystyle C_0 := \left (  \underset{X\in S_h|_K}{\sup} \left \| X \right \|^2_{g} \right )^2 \frac{1+A_0}{B_2}$.
\end{center}
Let $p\in K$ and $X,Y \in T^{1,0}_pM$ such that 
\begin{center}
    $\left \| X \right \|_{C_0h+g} = \left \| Y \right \|_{C_0h+g} = 1$.
\end{center}
We have
\begin{center}
    $\displaystyle \left (  \underset{X\in S_h|_K}{\sup} \left \| X \right \|^2_{g} \right )^2\left \| X \right \|^2_{h}\left \| Y \right \|^2_{h} \geq \left \| X \right \|^2_{g}\left \| Y \right \|^2_{g} $.
\end{center}
Hence
\begin{equation}
    \displaystyle C_0\left \| X \right \|^2_{h}\left \| Y \right \|^2_{h} \geq \left \| X \right \|^2_{g}\left \| Y \right \|^2_{g}\frac{1+A_0}{B_2} .
\end{equation}
On the other hand from inequality $(4)$ and since for any $\lambda>0$
\begin{center}
    $\displaystyle HBC(\lambda h)(X,Y) = \frac{R_{\lambda h}(X,\overline{X},Y,\overline{Y})}{\lambda^2 h(X,\overline{X})h(Y,\overline{Y})} 
    =  \frac{1}{\lambda}HBC(h)(X,Y)$
\end{center}
we have on $ K$
\begin{center}
    $HBC(C_0h+g)(X,Y) \leq$
    
    \vspace{0.2cm}
    $\left \| X \right \|^2_{C_0h}\left \| Y \right \|^2_{C_0h}HBC(C_0h)(X,Y) + \left \| X \right \|^2_{g}\left \| Y \right \|^2_{g}HBC(g)(X,Y)$
    
    \vspace{0.2cm}
    $= C_0 \left \| X \right \|^2_{h}\left \| Y \right \|^2_{h}HBC(h)(X,Y) + \left \| X \right \|^2_{g}\left \| Y \right \|^2_{g}HBC(g)(X,Y)$.
\end{center}
Since $HBC(h)\leq -B_2$ on $M$, and from inequality $(5)$  we have
\begin{center}
    $HBC(C_0h+g) \leq  -\left \| X \right \|^2_{g}\left \| Y \right \|^2_{g}(1+A_0)   + \left \| X \right \|^2_{g}\left \| Y \right \|^2_{g}HBC(g)$ on $K$.
\end{center}
Hence
\begin{center}
    $HBC(C_0h+g) \leq -\left \| X \right \|^2_{g}\left \| Y \right \|^2_{g}$.
\end{center}
Which implies  
\begin{center}
    $HBC(C_0h+g) \leq -C(g,C_0h)^2$ on $K$.
\end{center}
On $M\setminus K$ applying theorem $3.2$ gives
\begin{center}
    $\displaystyle HBC(C_0h+g) \leq -C(C_0h,g)^2\frac{B_2}{C_0}  -C(g,C_0h)^2B_1$.
\end{center}
Thus by taking 
\begin{equation}
    \displaystyle C_1 = \min \left\{ C(C_0h,g)^2\frac{B_2}{C_0}  +C(g,C_0h)^2B_1 , C(g,C_0h)^2\right\}
\end{equation}
we obtain on $M$

\begin{center}
    $HBC(C_0h+g) \leq -C_1$.
\end{center}

\end{proof}

\subsection{Quasi-bounded geometry and bounded domains}

Recall the definition of quasi-bounded geometry as in [WY20]. 

\begin{definition}
Let $N\geq 0$ be an integer, a Kähler manifold $(M,g)$ of dimension $n$  is said to be of \textit{$N$-quasi-bounded geometry}, if there exists $r_2>r_1>0$ such that for all $p\in M $ there exist a domain $U \subset \C^n$containing $0$, and a non-singular holomorphic map $\psi : U \rightarrow M$ satisfying the following conditions: 
 \begin{enumerate}
     \item $\displaystyle B(0,r_1) \subset  U  \subset B(0,r_2)$ with $\psi(0) = p$.
     \item There exists a constant $C_0>0$ depending uniquely on $r_1,r_2$ and $n$ such that \begin{center}
         $\displaystyle \frac{1}{C_0} h_{\C^n} \leq \psi^*g \leq C_0  h_{\C^n}$ on $U$
     \end{center}
     where $h_{\C^n}$ stands for the Euclidean metric on $\C^n$. 
     \item Let $(z^1,\ldots,z^n)$ be the standard holomorphic coordinates on $U \subset \subset \C^n$ and $g_{i\Bar{j}}$ the expression of $\psi^*g$ on those coordinates, then for any integer $1 \leq l \leq N$ there exist constants $C_l>0$ depending only on $r_1, r_2, n$ such that  \begin{center}
         $\displaystyle \underset{ U}{\sup} \left| \frac{\partial^{\left| \mu \right| + \left|\nu \right|}g_{i\Bar{j}} }{\partial z^{\mu}\partial \Bar{z}^\nu} \right| \leq C_l$ \end{center} for any multi-index $\mu$ and $\nu$ such that $\left|\mu \right| + \left|\nu \right| \leq l$.
     
 \end{enumerate}
 We say that $(M,g)$ is of \textit{$0$-quasi-bounded geometry} if it satisfies only conditions 1. and 2.
 \\
 \noindent And we say that $(M,g)$ is of \textit{quasi-bounded geometry} if, in addition, condition $3.$ is satisfied for any $l\geq1$.
\end{definition}

\noindent The following theorem can be found in ([WY20] theorem $9$). 

\begin{theorem}[Wu, Yau]
    Let $(M,g)$ be a complete Kähler manifold. Then $M$ is of quasi-bounded geometry if and only if for any integer $q\geq 0$ there exist constants $C_q >0$ such that  
\begin{center}
     $\displaystyle \underset{p\in M}{sup}\left|\nabla^q R_{i\Bar{j}k\Bar{l}}(p) \right|\leq C_q$.
\end{center}
 
\end{theorem}
\noindent From the proof of the preceding theorem in [WY20] (or from [TY90] proposition 1.2) one can straightforwardly deduce the following corollary.
\begin{corollary}
    
    Let $(M,g)$ be a complete Kähler manifold. Then $M$ is of $2$-quasi-bounded geometry if and only if there exists a constant $C >0$ such that  
\begin{center}
     $\displaystyle \underset{p\in M}{sup}\left| R_{i\Bar{j}k\Bar{l}}(p) \right|\leq C$.
\end{center}
 \end{corollary}
\noindent The lemma that follows appears in [WY20].

\begin{lemme}[Wu,Yau]
 Let $(M,g)$ be a Kähler manifold of $0$-quasi-bounded geometry then there exists a constant $C>0$ such that
 \begin{center}
     $ \displaystyle \mathfrak{K}_M (p, X ) \leq C \left \| X \right\|_{g}$ for $p\in M$ and $X \in T^{1,0}_p M$.
 \end{center}
 \end{lemme}
\noindent From that we deduce the following fact
\begin{lemme}
Let $M$ be a complex manifold, $h$ a Hermitian metric such that $HBC(h) \leq -B$ on $M$ for $B>0$, and $g$ a Kähler metric on $M$ of $0$-quasi-bounded geometry. Then there exists a positive constant $C>0$ such that
\begin{center}
    $C(g,h)\geq C$
\end{center}
\end{lemme}
\begin{proof}
    Let $p\in M$, $X\in T_p^{1,0}M$, since $HBC(h) \leq -B$ on $M$, we have by the Schwarz-Yau lemma (theorem 2.3 above)  
    \begin{center}
        $\displaystyle \left\|X \right\|_h \leq \sqrt{\frac{2}{B}}\mathfrak{K}_M(p,X)$.
    \end{center}
  On the other hand since $g$ is of $0$-quasi-bounded geometry, we have by lemma $3.4$ the existence of a constant $C_1>0$ such that
    \begin{center}
        $\left\|X \right\|_g \geq C_1\mathfrak{K}_M(p,X)$.
    \end{center}
    This implies
    \begin{center}
        $\displaystyle \frac{\left\|X \right\|^2_h }{\left\|X \right\|^2_g} \leq \frac{2}{C_1^2B}$.
    \end{center}
    Hence 
    \begin{center}
        $\displaystyle \frac{1}{1+\frac{\left\|X \right\|^2_h }{\left\|X \right\|^2_g}} \geq \frac{1}{1+ \frac{2}{C_1^2B}}$.
    \end{center}
    Hence
    \begin{center}
        $\displaystyle C(g,h) \geq \frac{1}{1+ \frac{2}{C_1^2B}}$.
    \end{center}
\end{proof}

\noindent The following proposition is very similar to ([CY80] proposition $1.4$). (Recall that given a bounded domain $\Omega \subset \subset \C^n$ we denote by $b^R_\Omega$ the Bergman metric of a ball centered at $0$, of radius $R$, containing $\Omega$ in its interior).
\begin{proposition}
Let $\Omega \subset \subset \C^n$ a bounded domain, $g$ a Kähler metric of $N$-quasi-bounded geometry on $\Omega$, then the metric $g + b^R_\Omega$ is of $N$-quasi-bounded geometry.  
\end{proposition}
\begin{proof}
Let $p \in \Omega$, since $g$ is of $N$-quasi-bounded geometry, we have, from the definition, a non singular holomorphic map $\phi : U \rightarrow \Omega$ with $\phi(0) = p$, and $ \displaystyle B(0,r_1) \subset  U  \subset B(0,r_2)$.
\\Denote by $V$ a neighborhood of $0\in \C^n$ such that $B(0,\frac{r_1}{4}) \subset V \subset B(0,\frac{r_1}{2}) \subset  B(0,r_1) \subset  U  \subset B(0,r_2)$ and $\psi : V \rightarrow \Omega$ the restriction of $\phi$ to $V$.
\\We will show that $\psi^*(g+b^R_\Omega)$ satisfies conditions 2. and 3. of the definition. 
\\Let us first remark the existence of a constant $B >0 $ depending uniquely on $R$ and such that
\begin{center}
     $\displaystyle \frac{1}{B} h_{\C^n} \leq b^R_\Omega \leq B  h_{\C^n}$ on $\Omega$.
\end{center}
Where $h_{\C^n}$ stands for the Euclidean metric on $\C^n$.
\\Moreover there exists a constant $C>0$ such that 
\begin{center}
         $\displaystyle \frac{1}{C} h_{\C^n} \leq \psi^*g \leq C  h_{\C^n}$ on $V$.
     \end{center}
Denote by $z =(z^1,\ldots,z^n)$ the standard coordinate system on $U \subset \C^n$, and $\psi(z) = (\psi_1, \ldots,\psi_n)$ the expression of $\psi(z)$ in the standard coordinate system of $\Omega \subset \C^n$.
\\Finally denote by $\Tilde{b}^R_{\Omega i\Bar{j}}$  (resp. $\Tilde{g}_{i\Bar{j}}$)  the expression of $\psi^*b^R_\Omega$  (resp. $\psi^*g$) in the coordinates $(z^1,\ldots,z^n)$ and $b^R_{\Omega k\Bar{l}}$ the expression of $b^R_\Omega$ in the standard coordinates of $\Omega \subset \C^n$. 
\\We have  
\begin{center}
    $\displaystyle \Tilde{b}^R_{\Omega i\Bar{j}} = \underset{k,l=1}{\overset{n}{\sum}} b^R_{\Omega k\Bar{l}} \frac{\partial \psi_k}{\partial z^i} \frac{\partial \Bar{\psi}_l}{\partial \Bar{z}^j} $
\end{center}
Let $(\lambda_1, \ldots, \lambda_n) \in \C^n$ we thus have
\begin{center}
    $\displaystyle \left| \underset{i,j=1}{\overset{n}{\sum}} (\Tilde{g}_{i\Bar{j}}+\Tilde{b}^R_{\Omega i\Bar{j}}) \lambda_i\overline{\lambda}_j\right|\leq C\underset{i=1}{\overset{n}{\sum}} \left|\lambda_i\right|^2 + \left|\underset{k,l=1}{\overset{n}{\sum}} b^R_{\Omega k\Bar{l}} \underset{i,j=1}{\overset{n}{\sum}} \frac{\partial \psi_k}{\partial z^i} \frac{\partial \Bar{\psi}_l}{\partial \Bar{z}^j}\lambda_i\overline{\lambda}_j \right|$
    \\$\displaystyle \leq C\underset{i=1}{\overset{n}{\sum}} \left|\lambda_i\right|^2 + \underset{i,j=1}{\overset{n}{\sum}} \left| \underset{k,l=1}{\overset{n}{\sum}}b^R_{\Omega k\Bar{l}} \frac{\partial \psi_k}{\partial z^i} \frac{\partial \Bar{\psi}_l}{\partial \Bar{z}^j} \right|\left| \lambda_i\overline{\lambda}_j\right|$
    \\$ \displaystyle\leq C\underset{i=1}{\overset{n}{\sum}} \left|\lambda_i\right|^2 +\underset{i,j=1}{\overset{n}{\sum}} \sqrt{\underset{k,l=1}{\overset{n}{\sum}}b^R_{\Omega k\Bar{l}} \frac{\partial \psi_k}{\partial z^i} \frac{\partial \Bar{\psi}_l}{\partial \Bar{z}^i}} \sqrt{\underset{k,l=1}{\overset{n}{\sum}}b^R_{\Omega k\Bar{l}} \frac{\partial \psi_k}{\partial z^j} \frac{\partial \Bar{\psi}_l}{\partial \Bar{z}^j}}\left|\lambda_i\overline{\lambda}_j \right| $
    \\$\displaystyle \leq C\underset{i=1}{\overset{n}{\sum}} \left|\lambda_i\right|^2 + B\underset{i,j=1}{\overset{n}{\sum}}  \sqrt{\underset{k=1}{\overset{n}{\sum}} \left| \frac{\partial \psi_k}{\partial z^i} \right|^2 } \sqrt{\underset{k=1}{\overset{n}{\sum}} \left| \frac{\partial \psi_k}{\partial z^j} \right|^2} \left| \lambda_i\overline{\lambda}_j\right | $
\end{center}
Since $ \displaystyle B(0,\frac{r_1}{4}) \subset V \subset B(0,\frac{r_1}{2})  \subset B(0,r_1) \subset  U  \subset B(0,r_2)$, and given $z_0\in V$ the polydisc $\displaystyle \ \Delta_0 =  \left\{z\in \C^n | \ \left|z^i-z_0^i \right| \leq \frac{r_1}{4} \right\} \subset U$, thus by Cauchy estimates  
\begin{center}
    $ \displaystyle \left| \frac{\partial \psi_k}{\partial z^j}(z_0) \right| \leq \frac{4\underset{\Delta_0}{\sup} \left|\phi_k \right|}{r_1}\leq \frac{4\underset{U}{\sup} \left|\phi_k \right|}{r_1}$
\end{center}
Let $P_k: \C^n \rightarrow \C $ the natural projection on the $k$-ith coordinate, since $\Omega$ is bounded, denote by 
\begin{center}
    $D_0 := \underset{1\leq k \leq n}{\sup} \underset{u\in P_k(\Omega ) }{\sup} \left| u \right| $.
\end{center}
It follows that for $1 \leq k,j\leq n$ and $z_0\in V$ 
\begin{center}
    $\displaystyle \left| \frac{\partial \psi_k}{\partial z^j}(z_0) \right| \leq \frac{4D_0}{r_1}$.
\end{center}
Thus on $V$
\begin{center}
    $\displaystyle \left| \underset{i,j=1}{\overset{n}{\sum}} (\Tilde{g}_{i\Bar{j}}+\Tilde{b}^R_{\Omega i\Bar{j}}) \lambda_i\overline{\lambda}_j\right|\leq C\underset{i=1}{\overset{n}{\sum}} \left|\lambda_i\right|^2 +  B\underset{i,j=1}{\overset{n}{\sum}} 16n\frac{D_0^2}{r_1^2} \left| \lambda_i\overline{\lambda}_j\right |  $
    \\$\displaystyle \leq (C+ 16Bn^2\frac{D_0^2}{r_1^2}) \underset{i=1}{\overset{n}{\sum}} \left|\lambda_i\right|^2 $
\end{center}
On the other hand $\psi^*b^R_\Omega$ is positive definite so
\begin{center}
    $\displaystyle \underset{i,j=1}{\overset{n}{\sum}} (\Tilde{g}_{i\Bar{j}}+\Tilde{b}^R_{\Omega i\Bar{j}}) \lambda_i\overline{\lambda}_j \geq \underset{i,j=1}{\overset{n}{\sum}} \Tilde{g}_{i\Bar{j}}\lambda_i\overline{\lambda}_j \geq \frac{1}{C}\underset{i=1}{\overset{n}{\sum}} \left|\lambda_i\right|^2 $
\end{center}
Which implies that $g + b^R_\Omega$ satisfies condition $2.$ of definition $3.2$ for the map $\psi$.
\\To show condition 3. let's denote by $(u_1,\ldots, u_n)$ the standard coordinate system on $\Omega \subset \C^n$, and remark that given an integer $l>0$ there exist constants $B_l>0$ such that for every multi-index $\mu$ and $\nu$ such that $\left|\mu \right| + \left|\nu \right| \leq l$ and for every $1 \leq k,r \leq n$ we have 
\begin{center}
         $\displaystyle \underset{ \Omega}{\sup} \left| \frac{\partial^{\left| \mu \right| + \left|\nu \right|}b^R_{\Omega k\Bar{r}} }{\partial u^{\mu}\partial \Bar{u}^\nu} \right| \leq B_l$
\end{center} 
Moreover given an integer $l>0$ there exist constants $A_l>0$ such that for every multi-index $\mu$ and $\nu$ such that $\left|\mu \right| + \left|\nu \right| \leq l$ and for every $1 \leq i,j \leq n$
\begin{center}
         $\displaystyle \underset{ V}{\sup} \left| \frac{\partial^{\left| \mu \right| + \left|\nu \right|}\Tilde{g}_{i\Bar{j}} }{\partial z^{\mu}\partial \Bar{z}^\nu} \right|\leq \underset{ U}{\sup} \left| \frac{\partial^{\left| \mu \right| + \left|\nu \right|}\Tilde{g}_{i\Bar{j}} }{\partial z^{\mu}\partial \Bar{z}^\nu} \right| \leq A_l$ \end{center}
Hence on $ V$
\begin{center}
    $\displaystyle \left| \frac{\partial^{\left| \mu \right| + \left|\nu \right|}(\Tilde{g}_{i\Bar{j}}+\Tilde{b}^R_{\Omega i\Bar{j}})}{\partial z^{\mu}\partial \Bar{z}^\nu} \right| \leq A_l + \left| \frac{\partial^{\left| \mu \right| + \left|\nu \right|}\Tilde{b}^R_{\Omega i\Bar{j}} }{\partial z^{\mu}\partial \Bar{z}^\nu} \right| $
    \\$\displaystyle =A_l + \left| \underset{k,r=1}{\overset{n}{\sum}}\frac{\partial^{\left| \mu \right| + \left|\nu \right|} }{\partial z^{\mu}\partial \Bar{z}^\nu}  (b^R_{\Omega k\Bar{r}} \frac{\partial \psi_k}{\partial z^i} \frac{\partial \Bar{\psi}_r}{\partial \Bar{z}^j}) \right| $
\end{center}
Thus using Leibniz formula and a similar argument as before with higher order Cauchy estimates, it follows that for any integer $ 0<l\leq N$ there exist constants $C_l'>0$ depending only on $D_0$, $r_1$, $B_l$ and $n$ such that
\begin{center}
    $\displaystyle \underset{ V}{\sup} \left| \frac{\partial^{\left| \mu \right| + \left|\nu \right|}(\Tilde{g}_{i\Bar{j}}+\Tilde{b}^R_{\Omega i\Bar{j}})}{\partial z^{\mu}\partial \Bar{z}^\nu} \right| \leq A_l + C_l'$
\end{center}
Which implies that $g+b^R_\Omega$ satisfies condition 3. and is of $N$-quasi-bounded geometry.
\end{proof}

\begin{theorem} Let $\Omega \subset \subset \C^n$ be a bounded domain, and $K\subset \Omega$ a compact subset. Let $g$ be a complete Kähler metric on $\Omega$ of negatively pinched holomorphic bisectional curvature on $\Omega\setminus K$, then there exists a positive constant $C_0 >0 $  such that $g+ C_0b^R_\Omega$ is of negatively pinched holomorphic bisectional curvature on $\Omega$.
\end{theorem}
\begin{proof}
From the hypothesis there exists a compact $K\subset \Omega$, and $B_1',B_1>0$ positive constants such that
    \begin{center}
        $-B_1'\leq HBC(g) \leq -B_1$ on $\Omega\setminus K$.
    \end{center} 
    As in corollary $3.2.1$ let $A_0 := \underset{X,Y\in S_g|_K }{\sup} \ HBC(g)(X,Y)$. If $A_0<0$ there is nothing to prove, so we can suppose $A_0\geq 0$. Since $b^R_\Omega$ is the restriction of the Bergman metric of $B(0,R)$ to $\Omega$ there exist some positive constants $B_2',B_2>0$ such that  
    \begin{center}
        $-B_2'\leq HBC(b^R_\Omega) \leq -B_2$ on $\Omega$.
    \end{center}
    By corollary $3.2.1$ there exist some non negative constants $C_0,C_1\geq0$ such that 
\begin{center}
    $HBC(C_0b^R_\Omega+g) \leq -C_1$
\end{center}   
with
\begin{center}
    $\displaystyle C_0 = \left (  \underset{X\in S_{b^R_\Omega}|_K}{\sup} \left \| X \right \|^2_{g} \right )^2 \frac{1+A_0}{B_2} $
\end{center}
and 
\begin{center}
    $ \displaystyle C_1 = \min \left\{ C(C_0b^R_\Omega,g)^2\frac{B_2}{C_0}  +C(g,C_0b^R_\Omega)^2B_1 , C(g,C_0b^R_\Omega)^2\right\}$.
\end{center}
Since we can suppose $A_0\geq 0$ notice that $C_0>0$. 
\\We will show that $C_1>0$.  Since $g$ is complete and of negatively pinched holomorphic bisectional curvature outside a compact, its curvature tensor is bounded, hence by corollary $3.3.1$ $g$ is of $2$-quasi-bounded geometry. Thus by lemma $3.5$ we have
\begin{center}
    $\displaystyle C(g,C_0b^R_\Omega) >0$ 
 and  $\displaystyle C(C_0b^R_\Omega,g)^2\frac{B_2}{C_0}  +C(g,C_0b^R_\Omega)^2B_1 >0$
\end{center}
which implies that $C_1>0$. 
\\Finally, by proposition $3.1$ the metric $g+C_0b^R_\Omega$ is of $2$-quasi-bounded geometry and clearly complete, hence we have, by corollary $3.3.1$, that the curvature tensor of $g+C_0b^R_\Omega$ is bounded, thus there exists a constant $C'_1>0$
\begin{center}
    $-C'_1 \leq HBC(g+C_0b^R_\Omega) \leq -C_1$ on $\Omega$.
\end{center}

\end{proof}

\section{Strictly pseudoconvex domains and squeezing function}
Let's consider the case of strictly pseudoconvex bounded domains of $\C^n$. First recall the theorem due to Ohsawa [Ohs81]  
\begin{theorem}[Ohsawa]
    Let $\Omega \subset \subset \C^n$ be a bounded pseudoconvex domain with $\mathcal{C}^1$ boundary the Bergman metric $b_\Omega$ of $\Omega$ is complete. 
\end{theorem}
\noindent The following result is due to Klembeck (in the case of smooth boundary) in [Kle78] and Kim and Yu (in the case of $\mathcal{C}^2$ boundary) in [KY96]. We rephrase it as follows  
\begin{theorem}[Klembeck/ Kim,Yu]
    Let $\Omega \subset \subset \C^n$ be a strictly pseudoconvex bounded domain with $\mathcal{C}^2$ boundary. There exist a compact $K\subset \Omega$, and constants $B',B>0$ such that
    \begin{center}
        $-B'\leq HBC(b_\Omega) \leq -B$ on $\Omega\setminus K$.
    \end{center}
\end{theorem}
\noindent Hence by theorem $3.6$ we obtain the corollary 
\begin{corollary}
    Let $\Omega \subset \subset \C^n$ be a strictly pseudoconvex bounded domain with $\mathcal{C}^2$ boundary, then there exists a complete Kähler metric on $\Omega$ with negatively pinched holomorphic bisectional curvature on $\Omega$.
\end{corollary}
\noindent We move to the case of the Kähler-Einstein metric and the squeezing function. First recall the classical fact that can be found in [MY83].

\begin{theorem}[Mok, Yau]
    Let $\Omega\subset \subset \C^n$ be a bounded pseudoconvex domain then there exists a unique (up to scaling ) complete Kähler-Einstein metric on $\Omega$.  
\end{theorem}
\noindent We will denote this metric by $g_{KE}$. We have the following theorem in [Gon18] 
\begin{theorem}[Gontard]
    Let $\Omega\subset \subset \C^n$ be a pseudoconvex bounded domain, let $q\in \partial \Omega$ such that the squeezing function tends to $1$ when $z \rightarrow q$, then for any $X,Y \in T^{1,0}_z\Omega$
    \begin{center}
        $\displaystyle HBC(X,Y) -\left ( 1 +\frac{\left | g_{KE}(X,\overline{Y})\right |^2}{\left \| X \right \|_{g_{KE}}^2\left \| Y \right \|_{g_{KE}}^2}  \right ) \rightarrow 0 $ when $z\rightarrow q$.
    \end{center}
\end{theorem}
\begin{remark}
   Remark that when the squeezing function tends to 1 at every point of the boundary, the pseudoconvexity assumption becomes redundant (this remark is usually taken for granted, we provide some details on that matter).
    \\Indeed, since the squeezing function of $\Omega$ is continuous (see [DGZ12] Theorem $3.1$). If $\underset{z\rightarrow q}{\lim}s_\Omega (z) = 1$ for every $q\in \partial\Omega$ then there exist a compact $K\subset \Omega$ and a real number $0<a\leq 1$ such that 
    \begin{center}
        $\displaystyle \underset{z\in \Omega \setminus K}{\inf}s_\Omega(z) \geq a >0$.
    \end{center}
On the other hand notice that for any $z\in \Omega$ one has 
\begin{center}
    $\displaystyle s_\Omega(z) \geq \frac{d(z,\partial\Omega)}{diam(\Omega)}$
\end{center}
where $d(\cdot,\cdot)$ stands for the Euclidean metric of $\C^n$. 
\\Hence
\begin{center}
    $\displaystyle \underset{z\in K}{\inf}s_\Omega(z) \geq \underset{z\in K}{\inf} \frac{d(z,\partial\Omega)}{diam(\Omega)} > 0$.
\end{center}
Thus 
\begin{center}
    $\displaystyle \underset{z\in \Omega}{\inf}s_\Omega(z)>0$.
\end{center} 
From a theorem of Yeung in [Yeu09] (or [DGZ12] corollary $4.6$ ) this implies that $\Omega$ is pseudoconvex.
\end{remark}

\begin{corollary}
Let $\Omega\subset \subset \C^n$ be a bounded domain, if the squeezing function tends to $1$ at every point of the boundary, then $\Omega$ admits a complete Kähler metric with negatively pinched holomorphic bisectional curvature.
\end{corollary}
\begin{proof}
     By theorem $4.4$ there exists a compact outside of which the complete Kähler-Einstein metric $g_{KE}$ is of negatively pinched holomorphic bisectional curvature, we can then apply theorem $3.6$.
\end{proof}

\vspace{1cm}
\noindent Omar Bakkacha : Univ. Grenoble Alpes, CNRS, IF, 38000 Grenoble, France.
\\E-mail address : omar.bakkacha@univ-grenoble-alpes.fr

\end{document}